  \documentclass [final]{amsart}
\usepackage{mathabx}
\usepackage[utf8]{inputenc} 
\usepackage[english]{babel}
\usepackage{amstext}
\usepackage{euscript}
\usepackage{bbm}
\usepackage{mathrsfs}
\usepackage{enumerate}
\usepackage{graphicx}
\usepackage{caption}
\usepackage{amscd}
\usepackage{verbatim}
\usepackage{amssymb}
\usepackage{amsthm}
\usepackage{amsmath}
\usepackage{amstext}
\usepackage{amsfonts}
\usepackage{latexsym}
\usepackage{mathtools} 
\usepackage{enumitem} 
\usepackage{accents} 
\usepackage{orcidlink,todonotes}
\usepackage[foot]{amsaddr} %package that puts the authors addresses on the title page as a footnote (when using AMSART document class only).
\usepackage{textcomp} % this package adds the section symbol with the \textsection command
\usepackage{cite} % this package makes a list of citations appear in a numerically ordered list and ...

\usepackage{hyperref}
\usepackage{todonotes}
\usepackage[color]{showkeys}
\newcommand{\eqdef }{\overset{\mbox{\tiny{def}}}{=}}

\newcommand{\utilde}[1]{\underaccent{\tilde}{ 1}}

\newcommand{\secref}[1]{\S\ref{ 1}}

\newcommand{\dbar}[1]{\bar{\bar{ 1}}}

\def\norm 1{\big\Vert 1\big\Vert}

\def\seq 1{\left< 1\right>}\def\mseq 1{\left< \! 1\!\right>}
\def\sep 1{\left( 1\right)}\def\adf 1{\left[ 1\right]}

\theoremstyle{definition}
\newtheorem{theorem}{Theorem}

\newtheorem{lemma}[theorem]{Lemma}
\newtheorem{proposition}[theorem]{Proposition}
\newtheorem{remark}[theorem]{Remark}
\newtheorem{definition}[theorem]{Definition}

\numberwithin{equation}{section}
\numberwithin{theorem}{section}
\begin{document}

\keywords{Boltzmann Equation, Landau Equation, Collisional Kinetic Theory.}
\subjclass[2020]{Primary 35Q20}

\title[axisymmetric solutions to the homogeneous Landau equation]{On the axially symmetric solutions to the spatially homogeneous Landau equation}

\author[J. W. Jang]{Jin Woo Jang$^{\dagger}$
\orcidlink{0000-0002-3846-1983}}
\address{$^{\dagger}$Department of Mathematics, Pohang University of Science and Technology (POSTECH), Pohang, South Korea (37673). \href{mailto:jangjw@postech.ac.kr}{jangjw@postech.ac.kr}  (\orcidlink{0000-0002-3846-1983} \href{https://orcid.org/0000-0002-3846-1983}{https://orcid.org/0000-0002-3846-1983})}
\thanks{$^\dagger$Supported by the National Research Foundation of Korea (NRF) grant funded by the Korean government (MSIT) NRF-2022R1G1A1009044, RS-2023-00210484, RS-2023-00219980 and by the Basic Science Research Institute Fund of Korea NRF-2021R1A6A1A10042944}

\author[J. Kim]{Junha Kim $^{\ddagger}$
\orcidlink{0000-0002-5962-6353}}
\address{$^\ddagger$ Department of Mathematics, Ajou University, Suwon, Gyeonggi-do, South Korea (16499). \href{mailto:}{
junha02@ajou.ac.kr} (\orcidlink{0000-0002-5962-6353} \href{https://orcid.org/0000-0002-5962-6353}{https://orcid.org/0000-0002-5962-6353})}
\thanks{$^\ddagger$Supported by the National Research Foundation of Korea(NRF) grant funded by the Korea government(MSIT) (No. RS-2024-00360798).}

\begin{abstract}
In this paper, we consider the spatially homogeneous Landau equation, which is a variation of the Boltzmann equation in the grazing collision limit. For the Landau equation for hard potentials in the style of Desvillettes-Villani \cite{MR1737547,MR1737548}, we provide the proof of the existence of axisymmetric measure-valued solution for any axisymmetric $\mathcal{P}_p(\mathbb{R}^3)$ initial profile for any $p\ge 2$. Moreover, we prove that if the initial data is not a single Dirac mass, then the solution instantaneously becomes analytic for any time $t>0$ in the hard potential case.  
In the soft potential and the Maxwellian molecule cases, we show that there are no solutions whose support is contained in a fixed line even for any given line-concentrated data.
\end{abstract}

\thispagestyle{empty}

\maketitle
%\tableofcontents

\section{Introduction}

The spatially homogeneous Landau equation is a fundamental model in kinetic theory, 
describing the evolution of particle distributions under grazing collisions. This paper focuses 
on proving the existence of weak solutions that exhibit axisymmetry when the initial data is 
axisymmetric.
For $t\in [0,T]$ and $v\in\mathbb{R}^3$ for some $T>0$, the spatially homogeneous Landau kinetic equation is given by
\begin{equation}\label{eq:L}
\partial_t f(t, v) = Q(f, f) = \partial_i \int_{\mathbb{R}^3} a_{ij}(v - v_*) \left( \partial_j f(v) f(v_*) - \partial_j f(v_*) f(v) \right) dv_*,
\end{equation}where we abbreviate $f(t,v)$ as $f(v)$.
Here, the kernel $a_{ij}$ is defined as:
\begin{equation}\label{def_a}
a_{ij}(z) = \left( \delta_{ij} - \frac{z_i z_j}{|z|^2} \right) |z|^{\gamma+2}, \quad \gamma \in [-3, 1],
\end{equation}with the standard Kronocker delta notation $\delta_{ij}=1\text{ if }i=j,\text{ and }=0,$ otherwise. We will call the case with $\gamma\in(0,1]$, $\gamma=0$, and $\gamma\in[-3,0)$ as the hard potential case, the Maxwellian moleculs, and the soft potential case, respectively. 

If we define $$b_i(z) := \partial_j a_{ij}(z) = -2z_i |z|^{\gamma}, \qquad c(z) := \partial_i \partial_j a_{ij}(z) = -2(\gamma+3)|z|^{\gamma}$$ and $$\overline{a_{ij}(v)} := a_{ij} * f, \qquad \overline{b_{i}(v)} := b_{i} * f, \qquad \overline{c(v)} := c * f,$$ we can rewrite the Landau equation \eqref{eq:L} as $$\partial_t f = \nabla \cdot (\overline{a} \nabla f - \overline{b} f) \qquad \mbox{or} \qquad \partial_t f = \overline{a_{ij}} \partial_i \partial_j f - \overline{c}f.$$ Note that $|a_{ij}(z)|\le |z|^{2+\gamma}$ and $|b_i(z)|\le |z|^{1+\gamma}.$

Our aim is to study the existence and regularity of axisymmetric weak solutions to \eqref{eq:L} for given initial data in probability measures on $\mathbb{R}^3$. 
\subsection{Notations}
In this paper, we denote the measure $f(t,v)dv$ as $f_t(dv)$. We first define the total density, the total energy, and the total entropy of the measure $f_t(dv)$ for some $t\geq0$ as 
\begin{equation*}
	M(f_t) := \int_{\mathbb{R}^3}  1f_t(dv), \qquad E(f_t) := \frac {1}{2} \int_{\mathbb{R}^3} |v|^2 f_t(dv),
\end{equation*} and 
\begin{equation*}
	\text{H}(f_t) := \int_{\mathbb{R}^3} f(t,v) \log f(t,v) \,\mathrm{d}v,
\end{equation*} respectively. 
%Additional notations include:
%\begin{align}\label{def_bc}
%b_i(z) & = \partial_j a_{ij}(z) = -2 z_i |z|^{\gamma}, \\
%c(z) & = \partial_i \partial_j a_{ij}(z) = -2(\gamma + 3)|z|^{\gamma}.
%\end{align}
Furthermore, we introduce the $s$-th moment
\begin{equation*}
	M_s(f_t) := \int_{\mathbb{R}^3} (1+|v|^2)^{\frac {s}{2}} |f(t,v)| \,\mathrm{d}v, \qquad s > 0.
\end{equation*}For each $t\ge 0,$ we define the following standard weighted function-space norms:
\begin{equation*}
	\| f_t \|_{L^p_s} := \left( \int_{\mathbb{R}^3} (1+|v|^2)^{\frac {s}{2}} |f(t,v)|^p \,\mathrm{d}v \right)^{\frac {1}{p}}, \qquad p\in [1,\infty), \,\,\,  s \ge 0,
\end{equation*}
\begin{equation*}
	\| f_t \|_{H^k_s} := \sum_{|\alpha|\le k}\left( \int_{\mathbb{R}^3} (1+|v|^2)^{\frac {s}{2}} |\nabla^\alpha_v f(t,v)|^2 \,\mathrm{d}v \right)^{\frac {1}{2}}, \qquad k \in \mathbb{N}, \,\,\,  s \ge 0.
\end{equation*} We also define the following H\"older space whose norm definition was introduced by \cite{MR181836,ladyzhenskaia1968linear}; define the H\"older space $\mathcal{H}^\ell([T_1, T_2] \times \Omega)$ ($T_2 > T_1 > 0, \, \Omega \text{ open set in } \mathbb{R}^3, \, \ell > 0, \, \ell \notin \mathbb{N}$), whose norm is given by 
\begin{multline}\label{holder def}\|f\|_{\mathcal{H}^\ell} = \sup_{T_1 < t < T_2, \, v \in \Omega} \sum_{|\alpha| + 2r \leq \lfloor \ell \rfloor} |\partial_t^r \partial_v^\alpha f(t, v)| \\ + \sup_{T_1 < t < T_2, \, v \neq w} \sum_{|\alpha| + 2r = \lfloor \ell \rfloor} \frac{|\partial_t^r \partial_v^\alpha f(t, v) - \partial_t^r \partial_v^\alpha f(t, w)|}{|v - w|^{\ell - \lfloor \ell \rfloor}} \\ + \sup_{s \neq t, \, v \in \Omega} \sum_{|\alpha| + 2r = \lfloor \ell \rfloor} \frac{|\partial_t^r \partial_v^\alpha f(s, v) - \partial_t^r \partial_v^\alpha f(t, v)|}{|t - s|^{(\ell - \lfloor \ell \rfloor)/2}}, \end{multline}where $\lfloor \ell \rfloor$ denotes the integer part of $\ell$ and $\alpha \in \mathbb{N}^3$.
We also define the mixed-H\"older norm $\mathcal{H}^{l_1,l_2}_{t,x}([T_1, T_2] \times \Omega):=\mathcal{H}^{l_1}_t([T_1,T_2];\mathcal{H}_x(\Omega))$ in the same spirit.

Throughout the paper, we will consider a measure-valued solution to the Landau equation. Thus we first introduce the following space of probability measures; let $\mathcal{P}(\mathbb{R}^3)$ be the set of probability measures on $\mathbb{R}^3$ 
, and let $$\mathcal{P}_p(\mathbb{R}^3) := \left\{ f \in \mathcal{P}(\mathbb{R}^3) ; \int_{\mathbb{R}^3} |v|^p f(dv) < \infty \right\},$$ for $p>0$. For $p\ge 1, $ it is endowed with the $p$-Wasserstein distance, defined as $$\mathcal{W}_p(\mu,\nu)\eqdef \left(\inf_{\gamma\in \Pi(\mu,\nu)}\int_{\mathbb{R}^3\times\mathbb{R}^3}\text{dist}(x,y)^p d\gamma(x,y)\right)^{\frac{1}{p}}.$$ For $p\in(0,1)$, one can still define $\mathcal{W}_p$ by removing the power $1/p$ from the infimum.

Here we introduce our notion of measure-valued weak solutions in the hard potential case.
\begin{definition}\label{weak def}
Let $\gamma \in (0,1]$. We say $f(t)$ is a weak solution to the Landau equation if
\begin{equation*}
	f \in L^{\infty}_{loc}((0,\infty);\mathcal{P}_{2}(\mathbb{R}^3)) \cap L^1_{loc}((0,\infty);\mathcal{P}_{2+\gamma}(\mathbb{R}^3)),
\end{equation*}
\begin{equation*}
	\frac {1}{2} \int_{\mathbb{R}^3} |v|^2 f_t(dv) \,\mathrm{d}v \leq \frac {1}{2} \int_{\mathbb{R}^3} |v|^2 f_0(dv), \qquad t \geq 0,
\end{equation*} and 
\begin{multline*}
		\int_{\mathbb{R}^3} \varphi(v) f_T(dv) - \int_{\mathbb{R}^3} \varphi(v) f_0(dv) 
		= \int_0^T \int_{\mathbb{R}^3} \int_{\mathbb{R}^3} a_{ij}(v-v_*) \partial_i \partial_j \varphi(v) f_t(dv_*) f_t(dv) \mathrm{d}t \\
		+ 2 \int_0^T \int_{\mathbb{R}^3} \int_{\mathbb{R}^3} b_{i}(v-v_*) \partial_i \varphi(v) f_t(dv_*) f_t(dv) \mathrm{d}t
	\end{multline*} for all $\varphi \in C^2_{b}(\mathbb{R}^3)$ and $T \geq 0$. 
\end{definition}
In addition, for the soft potential case with $\gamma<-1,$ we define the space
\begin{equation*}
    \mathcal{J}_{\alpha} := \left\{ f \in \mathcal{P}(\mathbb{R}^3) : \sup_{v \in \mathbb{R}^3} \int_{\mathbb{R}^3} |v-v_*|^{\alpha} f(dv_*) < \infty \right\},
\end{equation*}with $\alpha \in (-3,0)$, and in the definition of the weak solution above, we additionally require that $f \in L^1_{loc}((0,\infty);\mathcal{J}_{\gamma+1})$ for $\gamma \in [-3,-1)$. This additional condition for the soft potentials is needed to make the weak formula well-defined. 

It is well-known that the Landau equation possesses special stationary solutions named the Maxwellian equilibria; for a constant local temperature $T>0$, define \begin{equation*}
    M_T(v) = (2\pi T)^{-\frac{3}{2}}e^{-\frac{|v|^2}{2T}}, \qquad T>0.
\end{equation*}
A Dirac mass $\delta_0$, the limit case of $M_T(v)$ with $T=0$, is also a unique weak solution to \eqref{eq:L} in the distribution sense.

\subsection{Previous Results}
The study of the spatially homogeneous Landau equation has seen significant progress in the last few decades. The Landau equation is an approximated model of the classical Boltzmann equation without angular cutoff. It has been shown in the spatially homogeneous case that the Boltzmann collision operator converges to the Landau operator in the grazing collision limit for the Coulombic interaction. Here we introduce some convergence results from the Boltzmann equation for inverse-power-law potentials to the Landau equation in the grazing collision limit: 
\begin{itemize}
	\item Degond and Lucquin \cite{MR1167768} studied the limiting process involved, called the \textit{asymptotics of grazing collisions}.
	\item Arsen’ev and Buryak \cite{MR1055522} proved the convergence of solutions of the Boltzmann equation towards solutions of the Landau equation when grazing collisions prevail. 
	\item Desvillettes \cite{MR1165528} gave a mathematical framework for more physical situations, but excluding the main case of Coulomb potential.
	\item Villani \cite{MR1650006} derived rigorously the Landau equation for the Coulomb potential
    \item Also, for more recent results, see \cite{MR4706027,MR3505177,MR3180973}. We also mention a recent work on the weak convergence of the quantum Boltzmann equation to the Landau-Fokker-Planck equation \cite{MR4287184} in the semi-classical limit. See \cite{he2024semi} for the asymptotic expansion.
\end{itemize}

Regarding the existence theory for the spatially homogeneous Landau equation, we briefly mention the works \cite{MR1737547,MR1737548,FOURNIER20211961} and \cite{MR1646502} for hard potentials and the Maxwellian molecule case, respectively. For moderately soft potentials with $\gamma \in [-2,0),$ we introduce the works \cite{MR3884792,MR3158719,MR3582250,MR2502525} on the global wellposedness theory. For very soft potentials, it has been a long open problem to prove the global wellposedness due to high singularity of the reaction term. For the case with $\gamma>-3,$ Carrillo-Delgadino-Desvillettes-Wu \cite{MR4746872} proposed the gradient flow perspective for the Landau equation by constructing a non-local metric with respect to which the Landau equation can be viewed as a gradient flow. Very recently, Guillen and Silvestre \cite{2311.09420} proved the global existence of a unique classical solution by means of the Fisher information; see Theorem \ref{GSglobal}. See also other contributions \cite{MR3614751,MR4735821}. For the isotropic Landau equation in the style of Krieger-Strain \cite{MR2901061}, we introduce  the construction of  long-time radial solutions by Gualdani-Guillen \cite{MR3599518}. We also list the results on the spatially homogeneous relativistic Landau equation \cite{MR4042221,MR3959729}.

Regarding the spatially inhomogeneous Landau equation, the papers \cite{MR1278244,MR1392006} introduce solutions in the renormalized sense; we mention the construction of the renormalized solution with a defect measure \cite{MR2037247}. In addition, there have been concrete development on the theory on the global wellposedness and regularity of classical or weak solutions nearby the Maxwellian or the vacuum equilibria and its asymptotic stability. These have been established for both Cauchy problems in the whole space and the boundary-value problems in bounded domains with additional self-consistent fields effects. See \cite{1905.00173,MR2904573,Guo-L,KGH,MR3625186,MR3948345,MR2100057,MR2366140,MR2209761,MR4514055,2401.00554,MR4163828,MR2921603,MR4706027,MR4568409,MR4230064,MR4212191,MR4050580,MR3101794,MR2506070,MR4562582,MR3778645,MR4072211,MR4378091}.

\subsection{Conservation laws}
Formally, we can see from the equation that
	\begin{equation*}
	\begin{aligned}
		\frac {d}{dt} \int_{\mathbb{R}^3} f(v) \,\mathrm{d}v = \frac {d}{dt} \int_{\mathbb{R}^3} vf(v) \,\mathrm{d}v = 0.
	\end{aligned}
\end{equation*} In addition, we can also have
	\begin{equation}\label{energy con}
	\begin{aligned}
		\frac {d}{dt} \frac {1}{2} \int_{\mathbb{R}^3} |v|^2f(v) \,\mathrm{d}v &= \frac {1}{2} \int_{\mathbb{R}^3} \partial_i \left( \overline{a_{ij}(v)} \partial_j f(v) - \overline{b_i(v)} f(v) \right)|v|^2 \mathrm{d}v \\
		&= \int_{\mathbb{R}^3} \overline{a_{ij}(v)} \delta_{ij} f(v) \,\mathrm{d}v + 2 \int_{\mathbb{R}^3} \overline{b_i(v)} f(v) v_i \mathrm{d}v \\
		&=0,
	\end{aligned}
\end{equation} because from $b_i(v-v_*) = -b_i(v_*-v)$ and $a_{ii}(z) = -z_i b_i(z) = 2|z|^{2+\gamma}$, and
\begin{equation*}
	\begin{aligned}
	2 \int_{\mathbb{R}^3} \overline{b_i(v)} f(v) v_i \mathrm{d}v &= \int_{\mathbb{R}^3} \int_{\mathbb{R}^3} (v-v_*)_i b_i(v-v_*)f(v_*)f(v) \,\mathrm{d}v_* \mathrm{d}v \\
		&= - \int_{\mathbb{R}^3} \left( \int_{\mathbb{R}^3} a_{ii}(v-v_*) f(v_*) \,\mathrm{d}v_*  \right) f(v) \,\mathrm{d}v.
	\end{aligned}
\end{equation*}

\subsection{Further properties of the Landau collision operator for hard potentials}We now focus on the hard potential case $(\gamma \in (0,1]$). In this subsection, we introduce two well-known properties of the Landau equation for hard potentials: gain of moments and ellipticity. 
\subsubsection{Gain of moments}
Here we introduce a theorem on the gain of moments for the hard potential case due to Desvillettes and Villani \cite{MR1737547}:
\begin{theorem}[Theorems 3 and 6 of \cite{MR1737547}]\label{Villani theorem}Let $f$ be any weak solution of the Landau equation \eqref{eq:L} with initial datum $f_{in}\in L^1_2(\mathbb{R}^3)$, satisfying the decay of energy 
$$E(f(t,\cdot))\le E(f_{in}(\cdot)),\text{ for }t\ge 0.$$
Then \begin{enumerate}
    \item For all $s>0,$ if $M_s(f_{in})<+\infty,$ then $\sup_{t\ge 0}M_s(f(t,\cdot))<+\infty,$ and for all $T>0,$ $$\int_0^TM_{s+\gamma}(f(t,\cdot))dt<+\infty.$$
    \item For all time $t_0>0,$ and all number $s>0,$ there exists a constant $C_{t_0}>0 $, depending only on $M_{in},$ $E_{in}$, and $t_0,$ such that for all time $t\ge t_0,$ $$M_s(f(t,\cdot))\le C_{t_0}.$$
    \item For any $t\ge 0,$ $E(f(t,\cdot))=E_{in}:$ the energy is automatically conserved.
    		\item If $f_{in}\in L^1_{2+\delta}$ for some $\delta>0$, there exists a weak solution starting from $f_{in}$.
		\item If $f_{in}\in L^1_{2+\delta}$ for some $\delta>0$ and $f_{in}$ is not concentrated on a line, there exists a weak solution starting from $f_{in}$ such that for any $t_0 > 0$, $\text{H}(f(t)) < \infty$ for all $t > t_0$ and 
		\begin{equation}\label{reg}
	\begin{aligned}
		\sup_{t \geq t_0} \| f(t) \|_{H^k_s} < \infty, \qquad k \in \mathbb{N}, \,\,\, s \geq 0.
	\end{aligned}
\end{equation}
\end{enumerate}

\end{theorem}The sketch of the proof is in Appendix \ref{app.gain of M}. 
 \subsubsection{Ellipticity of $\overline{a_{ij}(v)}$} In the same paper \cite{MR1737547}, it is also known the following ellipticity property of the Landau collision operator: 
 \begin{proposition}[Proposition 4 of \cite{MR1737547}]\label{prop.1.3}
     Let $\gamma \in (0,1]$. Let $T>0$ and $f \in L^{\infty}(0,T;L^1_2 \cap L \log L(\mathbb{R}^3))$ with $M(f(t)) = M(f_0)$, $E(f(t)) \leq E(f_0)$, and $\text{H}(f(t)) \leq \text{H}(f_0)$. Then, there exists a constant $K>0$ depending on $\gamma$, $M_0$, $E_0$, and $\text{H}_0$ such that $$\overline{a_{ij}(v)} \xi_i \xi_j \geq K(1+|v|^2)^{\frac {\gamma}{2}} |\xi|^2, \qquad v,\xi \in \mathbb{R}^3.$$
 \end{proposition}
	Regarding this proposition, we make the following remark:
	\begin{itemize}
	\item This ellipticity estimate is stronger than the uniform ellipticity. In the current paper, we will consider an approximated Landau equation with a regularized collision operator with $\overline{a_{ij}(v)}^{\varepsilon} = (a_{ij}(v) \mu_\varepsilon(v)) * f$, where $\mu_\varepsilon(v) \sim |v|^{-2-\gamma}$ when $|v| \sim \infty$. In this case, we only have the uniform ellipticity estimate without the gain of moments of order $\gamma$.
	\item Regarding the counterpart of the soft potential case with $\gamma\in(-3,0), $ some a priori estimates for the soft potential are obtained by Alexandre, Lin, and Liao \cite[Proposition 2.1]{MR3375485}.
	\end{itemize}It can be also shown that the entropy condition $f\in L^\infty(0,T;L\log L(\mathbb{R}^3))$ is not essential. Namely, Desvillettes and Villani \cite{MR1737547} considered the following infinite entropy case in the same paper as well.
	\begin{lemma}[Lemma 9 of \cite{MR1737547}]\label{lem.elliptic.infinite}
	    Let $T>0$ and $f$ be a non-zero non-negative function $f$ in $ L^\infty([0,T];L^1_{2}(\mathbb{R}^3))\cap C([0,T];W^{-2,1}(\mathbb{R}^3))$ with $f_{in}\in L^1_2(\mathbb{R}^3)$. Then, there exist constants $\Delta t \in (0,T]$ and $K>0$ such that $$\overline{a_{ij}(v)} \xi_i \xi_j \geq K(1+|v|^2)^{\frac {\gamma}{2}} |\xi|^2, \qquad v,\xi \in \mathbb{R}^3, \,\, t \in [0,\Delta t].$$
	\end{lemma}

    \begin{remark}
		In \cite[Remarks of Lemma 9]{MR1737547}, it is stated that the proof of the lemma allows in fact the initial condition $f_{in}$ to be a measure, provided that it be not concentrated on a single line. In our case for the proof of Theorem \ref{main thm}, by our key lemma (Lemma \ref{lem_line}) we observe that the support of a solution $f$ instantaneously escapes a line support for any $t>0.$ Then for any $t_0>0$, on the time interval $[t_0,\infty)$ we can consider the measure $f_{t_0}$ as the initial profile of Lemma \ref{lem.elliptic.infinite} and obtain the ellipticity.    We introduce a sketch of their proof here. The main observation of the proof for Lemma \ref{lem.elliptic.infinite} in \cite{MR1737547} is that there exist three balls $B_i$, whose centers are not on a line such that $\displaystyle \int_{B_i} f_0(v) \,\mathrm{d}v > 0$. Then one can show that there exists $\Delta t > 0$ such that $\displaystyle \inf_{[0,\Delta t]} \int_{B_i} f(t,v) \,\mathrm{d}v > 0$.  Using this, one proves that for $|v|\le R$ for some $R>0$ and for $t\in[0,\Delta t],$ $$\bar{a}_{ij}(t,v)\xi_i\xi_j \ge \sin^2(\theta_0(R)) r^{\gamma+2}|\xi|^2,$$ for some angle $\theta_0(R)>0$ such that when $|v|\le R,$ for all $\xi\in\mathbb{R}^3,$ there exists $i\in\{1,2,3\}$ such that $D_{\theta_0(R),\xi}(v)\cap B_i$ is empty where $D_{\theta_0(R),\xi}(v)$ is a thin cone centered at $v$, of axis directed by $\xi$ and of angle $\theta_0(R)$, defined as
        $$D_{\theta_0(R),\xi}(v)\eqdef \left \{v_*\in\mathbb{R}^3:\left|\frac{v-v_*}{|v-v_*|}\cdot \xi \right|\ge \cos\theta\right\}.$$ For $|v|\ge R$ with $R$ sufficiently large, on the ball $B_i$ with empty intersection with $D_{\theta_0(R),\xi}(v)$, one can show that $\bar{a}_{ij}(t,v)\xi_i\xi_j \gtrsim |v|^\gamma|\xi|^2\int_{B_i}f(v_*)dv_*.$ Combining the two results, we obtain the ellipticity.
        \end{remark}

	%Note that $\int_0^{\Delta t} \int_{B_i} \partial_t f(v) \lesssim \int_0^{\Delta t} \int_{\mathbb{R}^3} \int_{\mathbb{R}^3} (1+|v|^{2+\gamma} + |v_*|^{2+\gamma}) f f_* \lesssim C \Delta t.$

\subsection{Our main result and relevant previous results}
Now we state our main theorem.
\begin{theorem}\label{main thm}
Let $\gamma \in (0,1]$ and the initial profile $f_0 \in \mathcal{P}_p(\mathbb{R}^3)$ for $p\ge 2$ which is an axisymmetric initial data. Then, there exists an axisymmetric measure-valued weak solution to the Landau equation $$f\in C([0,\infty);\mathcal{P}(\mathbb{R}^3)) \cap L^{\infty}((0,\infty);\mathcal{P}_{p}(\mathbb{R}^3)) \cap L^{1}_{loc}((0,\infty);\mathcal{P}_{p+\gamma}(\mathbb{R}^3)).$$ Moreover, if $f_0$ is not a single Dirac mass, then $f$ is analytic for all $t>0$.
\end{theorem}

\begin{remark}Here are some remarks on our main theorem. 
    \begin{enumerate}
        \item If $f_0 \in \mathcal{P}_p(\mathbb{R}^3)$ is axisymmetric with $p > 2$, then the axisymmetric function $f$ is a unique weak solution to the Landau equation starting from $f_0$; see \cite[Theorem 8]{FOURNIER20211961}. The uniqueness is given in terms of the optimal transportation cost $\mathcal{T}_p(f,\tilde{f})$.
        \item In the soft potential case, we can only show that there are no axisymmetric line-concentrated solutions even for any given line-concentrated data  (see Lemma \ref{lem_line}). In the case when $\gamma\in[-3,-2)$ it has been proved by Golse-Imbert-Ji-Vasseur in \cite{2206.05155} that if a weak solution is axisymmetric, then it is smooth away from the symmetry axis.
        \item For the moderately soft potential case for $\gamma \in (-2,0)$, 
Fournier and Gu\'erin proved the existence of weak solutions; see \cite[Corollary 4]{MR2502525}. 
    \end{enumerate}
\end{remark}

We then introduce relevant previous results for hard potentials. 
 In \cite{MR1737547}, for initial data in $L^1_{2+\delta} \cap L \log L$ with some $\delta > 0$, Desvillettes and Villani prove that there exists a global weak solution to the Landau equation. Moreover, the solution becomes smooth for $t > 0$; see Theorem \ref{Villani theorem}.
On the other hand, Chen-Li-Xu prove in \cite{MR2557895} the instantaneous analytic smoothing of solutions for hard potentials:
\begin{theorem}[Theorem 1 of \cite{MR2557895}]\label{analytic theorem}
Let $f_0$ be the initial datum with finite mass, energy and entropy and $f(t,v)$ be any solution of the Cauchy problem \eqref{eq:L} such that for all time $t_0>0$ and all integer $m\ge 0,$ $$\sup_{t\ge t_0}\|f(t,\cdot)\|_{H^m_\gamma}\le C,$$ with $C$ a constant depending only on $\gamma,M_0,E_0,H_0,m$ and $t_0.$ Then for all time $t>0,$ $f(t,v)$, as a real function of $v$ variable, is analytic in $\mathbb{R}^3_v$.
\end{theorem}
Regarding the probability measure-valued solutions, Fournier and Heydecker construct a weak solution in the hard potentials:
\begin{theorem}[Theorems 3 and 5 of \cite{FOURNIER20211961}]\label{fournier analytic}
	Let $\gamma \in (0,1]$ and $f_0 \in \mathcal{P}_2(\mathbb{R}^3)$. 
	\begin{itemize}
	\item There exists a weak solution to the Landau equation starting from $f_0$
	\item If $f_0$ is not a Dirac mass, then $f$ is analytic for all $t>0$.
	\end{itemize}
\end{theorem}
To prove the second statement, they prove the following lemma
\begin{lemma}[Lemma 16 of \cite{FOURNIER20211961}]
	Let $\gamma \in (0,1]$ and $f_0 \in \mathcal{P}_4(\mathbb{R}^3)$ which is not a single Dirac mass, and $f$ be a weak solution to the Landau equation starting from $f_0$. Then for any given $t_0>0$, there exists $t \in (0,\delta)$ such that $f(t)$ is not concentrated on any line $\ell \subset \mathbb{R}^3$.
\end{lemma}
In the same paper, the authors also provide the following uniqueness theorem in the sense of transportation cost $\mathcal{T}_p$:
\begin{theorem}
    \label{uniq.thm}
    Fix $\gamma\in(0,1]$ and $p>2$ and two weak solutions $f_t$ and $\tilde{f}_t$ to \eqref{eq:L} starting from $f_0$ and $\tilde{f}_0$, both belonging to $\mathcal{P}_p(\mathbb{R}^3).$ There is a constant $C,$ depending only on $p$ and $\gamma$, such that for all $t\ge 0,$ 
    \begin{multline*}
        \mathcal{T}_p(f_t,\tilde{f}_t)\\
        \le \mathcal{T}_p(f_0,\tilde{f}_0)\exp \left(C[1+\sup_{s\in[0,t]}M_p(f_s+\tilde{f}_s)][1+\int_0^t(1+M_{p+\gamma}(f_s+\tilde{f}_s)ds]\right).
    \end{multline*}
\end{theorem}

% Lastly, there is also a result on the \text{isotropic} Landau equation; 
% the \textit{isotropic} Landau equation is given by: $$\partial_t f = a[f] \Delta f + f^2, \qquad a[f] := \frac {1}{4\pi} \frac {1}{|v|} * f = \frac {1}{4\pi} \int_{\mathbb{R}^3} \frac {f(v_*)}{|v-v_*|} \,\mathrm{d}v_*.$$To this equation, we have the existence and the smoothing of radially symmetric solutions:
% \begin{theorem}[Theorem 1.1 in \cite{MR3599518}]
% Let $f_0$ be a radially symmetric and monotonically decreasing initial data such that $$\int_{\mathbb{R}^3} f_0(v) \,\mathrm{d}v, \,\, \int_{\mathbb{R}^3} |v|^2 f_0(v) \,\mathrm{d}v, \,\, \int_{\mathbb{R}^3} f_0(v) \log f_0(v) \,\mathrm{d}v < \infty.$$ Then, any solution $f$ to the the isotropic Landau equation satisfies $$\| f \|_{L^{\infty}(0,T;L^{\frac {3}{2}+})} < \infty,$$ and consequently, $f \in C^{\infty}((0,T) \times \mathbb{R}^3)$ for all $T>0$.
% \end{theorem}

We also introduce results for soft potentials briefly.  Theorem 3 of \cite{MR2502525} gives uniqueness of weak solutions for any $f_0 \in \mathcal{P}_2(\mathbb{R}^3)$ for the moderately soft potentials $\gamma\in(-2,0)$.  Existence of weak solution is also given as \cite[Corollary~4]{MR2502525} for $\gamma \in (-2,0)$ and $f_0 \in \mathcal{P}_q(\mathbb{R}^3)$ for sufficiently large $q$ and $\text{H}(f_0) < \infty$. 
Lastly, for the very soft potentials $\gamma<-2$, we also introduce a recent result by Guillen and Silvestre on the global wellposedness including the Coulombic case:
\begin{theorem}[Guillen-Silvestre \cite{2311.09420}]\label{GSglobal}
  Let $\gamma\in[-3,1].$ Let $f_0$ be initial data in $C^1$ and is bounded by a Maxwellian. Assume also that the initial Fisher information is bounded. Then there is a unique global classical solution $f$ to the Landau equation \eqref{eq:L} with the initial data $f(0,v)=f_0.$ For any positive time, this function $f$ is strictly positive, and bounded above by a Maxwellian. The Fisher information $i(f)$ is non-increasing.
\end{theorem}

This paper builds on these foundational results by specifically addressing the existence and regularity of axisymmetric measure-valued weak solutions when the initial data exhibits axial symmetry. In the next section, we introduce a key lemma which works as a crucial element in proving our main theorem.
\section{Key Lemma for our result}
A critical element in proving the instantaneous formation of analyticity in Theorem \ref{main thm}, is the following lemma:
\begin{lemma}\label{lem_line}
Let $\gamma\in[-3,1]$. Let $f\in L^1_{\text{loc}}((0,\infty);\mathcal{P}_{2+\gamma}(\mathbb{R}^3))$ be any weak solution to the Landau equation in the sense of Definition \ref{weak def} which is not a single Dirac mass. Let $\ell$ be any given line in $\mathbb{R}^3$. Then there is no time interval $[t_1,t_2]\in(0,\infty)$ with $t_1<t_2$ such that $f(t,\cdot)$ is concentrated on a line $\ell\in\mathbb{R}^3$ for a.e.  $t\in[t_1,t_2].$
\end{lemma}
\begin{remark}
Here are some remarks on the lemma above:
    \begin{enumerate}
        \item Lemma \ref{lem_line} is valid not only for the hard potentials $\gamma\in(0,1]$ but also for the soft potentials and the Maxwellian-molecule cases $\gamma\in[-3,0].$
        \item Lemma \ref{lem_line} is for a given fixed line $\ell$.
        \item The proof of Lemma is direct and is much simpler than that of \cite[Lemma 16]{FOURNIER20211961}. In addition, the requirement $\gamma\in(0,1]$ and $f_0\in\mathcal{P}_4(\mathbb{R}^3)$ is not necessary here. Lemma 16 of \cite{FOURNIER20211961} goes by:
\begin{lemma}[Lemma 16 of \cite{FOURNIER20211961}]
    Let $\gamma \in (0,1]$ and $f_0 \in \mathcal{P}_4(\mathbb{R}^3)$ which is not a single Dirac mass, and $f$ be a weak solution to the Landau equation starting from $f_0$. Then for any given $t_0>0$, there exists $t \in (0,\delta)$ such that $f(t)$ is not concentrated on any line $\ell \subset \mathbb{R}^3$.
\end{lemma}
\item Our lemma is for a fixed line $\ell \in \mathbb{R}^3.$ Even if the support instantaneously spans away from the line $\ell$ it is still possible that the support is in a different line (e.g., rotating lines) and this does not imply that the support contains three points not on a line. (In this sense, Lemma 16 of \cite{FOURNIER20211961} is more general as they prevent the support also from the rotating line.) Therefore, we consider the axisymmetric solution for the proof of the existence so that the support instantaneously contain a triangle. 
    \end{enumerate}
\end{remark}
\begin{proof}[Proof of Lemma \ref{lem_line}]Without loss of generality, we assume that $f$ is concentrated on the line $\ell = \{ v \in \mathbb{R}^3;v_2=v_3=0 \}$ for all $t \in [0,T]$ for some $T>0$. From the weak formula, we have
\begin{equation*}
	\begin{gathered}
		\int_{\mathbb{R}^3} \varphi(v) f_T(dv) - \int_{\mathbb{R}^3} \varphi(v) f_0(dv) \\
		= \int_0^T \int_{\mathbb{R}^3} \int_{\mathbb{R}^3} a_{ij}(v-v_*) \partial_i \partial_j \varphi(v) f_t(dv_*) f_t(dv) \mathrm{d}t \\
		+ 2 \int_0^T \int_{\mathbb{R}^3} \int_{\mathbb{R}^3} b_{i}(v-v_*) \partial_i \varphi(v) f_t(dv_*) f_t(dv) \mathrm{d}t =: I_{ij} + II_{i}.
	\end{gathered}
\end{equation*}
From the assumption, $f$ vanishes on $\mathbb{R}^3 \setminus \ell$. Recalling $$a_{ij}(v-v_*) := \left( \delta_{ij} - \frac {(v-v_*)_i(v-v_*)_j}{|v-v_*|^2} \right)|v-v_*|^{2+\gamma},$$ we can easily show that $I_{ij} = 0$ for $i \not= j$ and $I_{11} = 0$. Thus, we have $$I_{ij} = I_{22} + I_{33} = 2\int_0^T \int_{\mathbb{R}^3} \int_{\mathbb{R}^3} |v-v_*|^{2+\gamma} \Delta_h\varphi(v) f_t(dv) f_t(dv_*) \,\mathrm{d}t.$$ On the other hand,  for $v,v_* \in \ell$, $$b_i(v-v_*) = -2(v-v_*)_i|v-v_*|^{\gamma} = 0, \qquad i=2,\,3.$$ This implies $$II_{i} = II_1 = -4 \int_0^T \int_{\mathbb{R}^3} \int_{\mathbb{R}^3} (v-v_*)_1|v-v_*|^{\gamma} \partial_1 \varphi(v) f_t(dv_*) f_t(dv) \,\mathrm{d}t.$$ We arrived at 
	\begin{equation*}
	\begin{gathered}
		\int_{\mathbb{R}^3} \varphi(v) f_T(dv) - \int_{\mathbb{R}^3} \varphi(v) f_0(dv) \\
		= 2\int_0^T \int_{\mathbb{R}^3} \int_{\mathbb{R}^3} |v-v_*|^{2+\gamma} \Delta_h \varphi(v) f_t(dv) f_t(dv_*) \,\mathrm{d}t \\
		-4 \int_0^T \int_{\mathbb{R}^3} \int_{\mathbb{R}^3} (v-v_*)_1|v-v_*|^{\gamma} \partial_1 \varphi(v) f_t(dv_*) f_t(dv) \,\mathrm{d}t.
	\end{gathered}
\end{equation*} We take $\varphi = \frac {1}{4}(v_2^2 + v_3^2) \sigma(v_2,v_3)$, where $\sigma \in C^{\infty}_c(\mathbb{R}^2)$ is a smooth radial function such that $\sigma = 1$ on $B(0;\frac{1}{2})$ and $\sigma = 0$ on $\mathbb{R}^3 \setminus B(0;1)$. Then, $\partial_1 \varphi(v) = 0$ and $\partial_2 \varphi(v) = \partial_3 \varphi(v) = 0$ on $B(0;\frac{1}{2})$. Thus, our assumption gives that
	\begin{equation*}
	\begin{gathered}
		0 = \int_0^T \int_{\mathbb{R}^3} \int_{\mathbb{R}^3} |v-v_*|^{2+\gamma} f_t(dv) f_t(dv_*) \,\mathrm{d}t
	\end{gathered}
\end{equation*} which is a contradiction due to $f \in L^{\infty}((0,\infty);\mathcal{P}_2(\mathbb{R}^3))$ which is not a Dirac mass. This completes the proof. 
\end{proof}

In the next section, we prove our main theorem.

\section{Proof of Theorem \ref{main thm}}
This section is devoted to proving the main theorem, Theorem \ref{main thm}, on the existence of axisymmetric solutions to \eqref{eq:L} and its regularity. To this end, we first regularize the equation. We will consider the case with $p\ge 2$. We will also consider linear approximating equations to the regularized equation and consider the existence of iterated sequence of solutions to each linearized equation. On the other hand, we will show that the regularized operators commute with a rotation matrix $T$. We introduce each step below more in detail.
\subsection{Regularized System}
We first consider a regularized equation of the Landau equation \eqref{eq:L} as follows. Fix a constant $p\ge 2$ and let the initial profile $f_0\in \mathcal{P}_p(\mathbb{R}^3)$ be axially symmetric. Without loss of generality, let the initial total mass $M(f_0)$ be equal to 1. We now consider a regularized equation of \eqref{eq:L} as follows:\begin{equation}\label{eq.reg eq}
\partial_t f^\varepsilon = Q^\varepsilon(f^\varepsilon, f^\varepsilon) + \varepsilon  \Delta_v f^\varepsilon.
\end{equation}See also \cite{MR1055522,MR1737547} for similar types of regularization. Here, the collision operator $Q$ is now regularized as \begin{equation}\label{def_bar} 
Q^{\varepsilon}(g,f) := \overline{a_{ij}}^{\varepsilon,g} \partial_i \partial_j f - \overline{c}^{\varepsilon,g} f,
\end{equation} where $\overline{a_{ij}}^{\varepsilon,g} := a_{ij}^{\varepsilon} * g$, $\overline{c}^{\varepsilon,g}  := \partial_i \partial_j a_{ij}^{\varepsilon} * g$, and \begin{equation}
    \label{def.aijep}a_{ij}^{\varepsilon}(z) := \left( \delta_{ij} - \frac {z_iz_j}{|z|^2} \right)|z|^{2+\gamma} \mu_\varepsilon(|z|),
\end{equation} for some smooth function on $\mathbb{R}^3\setminus \{0\}$ such that
\begin{equation}\label{def.muep}
	\begin{aligned}
		\mu_\varepsilon(|z|) &=\begin{cases}
		    |z|^{-\gamma} \qquad\mbox{ on } B\left(0;\frac {\varepsilon}{2}\right), \\
            1 \qquad\qquad\mbox{ on } B(0;\varepsilon^{-1}) \setminus B(0; \varepsilon), \\
            |z|^{-2-\gamma} \qquad\mbox{ on } \mathbb{R}^3 \setminus B(0; 2\varepsilon^{-1}).
		\end{cases} 
	\end{aligned}
\end{equation} 
%One can see that \begin{equation}\label{aepcep estimate}
%    \| \nabla^k \overline{a_{ij}}^{\varepsilon,g} \|_{L^{\infty}} + \| \nabla^k \overline{c}^{\varepsilon,g} \|_{L^{\infty}} \leq C(k,\varepsilon,M(g(t))), \qquad k \in \mathbb{N} \cup \{ 0 \}.
%\end{equation} 
Indeed, in \cite[Eq.(23)]{MR1055522} we can see a simpler regularization of the equation as 
\begin{equation}\label{arsenev eq}\partial_t f^\varepsilon = Q(f^\varepsilon, f^\varepsilon) + \varepsilon \Delta_v f^\varepsilon ,\end{equation}even without regularizing the collision cross section. This is also a possible alternative. This also guarantees the strong ellipticity with the modulus of ellipticity is bounded from below by $\varepsilon.$ Then each linear approximating equation of this nonlinear equation has a unique classical solution given that the solution is bounded from above by $e^{a|v|^2}$ due to the classical parabolic theory \cite{ladyzhenskaia1968linear,MR199563}. More precisely, it is also possible to construct a sequence of linear solutions in the H\"older class $\mathcal{H}^{l/2,l}_{t,v}$ for $l\in (1,2)$ given that the initial data is in $\mathcal{H}^{l}_{v}$ and that the coefficients of the operator are in $\mathcal{H}^{l/2,l}_{t,v}$. In this paper, though working with \eqref{arsenev eq} is perfectly fine with the same proof below, we follow the regularization \eqref{eq.reg eq} instead and work with smoother coefficients. 
\subsubsection{Regularized initial profiles}
We also regularize the initial profile. We slightly modify the cutoff process done in \cite[page 1978, step 1]{FOURNIER20211961} by mollifying the cutoff initial data in addition. For a given axisymmetric initial data, we define for any given $n_0 \in \mathbb{N}$ a cutoff of the total mass $\alpha_{n} := \int_{\mathbb{R}^3} \mathbf{1}_{|v| \leq n}f_0(dv)$ for $n\ge n_0$, and define
\begin{equation}
    \label{initial}f^{n}_0 := (\alpha_n^{-1} \mathbf{1}_{|v| \leq n} f_0(dv)) * \sigma_{n}, \qquad n \geq n_0,
\end{equation} where $\sigma_{n} \in C^{\infty}_c(\mathbb{R}^3)$ is a non-negative radial function of approximation to the identity such that $\int_{\mathbb{R}^3} \sigma_{n}(v) \,\mathrm{d}v = 1$ for all $n \geq n_0$ and $\sigma_{n} \to \delta$ as $n \to \infty$.  Note that $M(f^n_0) = M(f_0) = 1$ and $E(f^n_0) \to E(f_0)$ as $n \to \infty,$ but $\text{H}(f^n_0) \to \infty$ could occur for some $f_0 \in \mathcal{P}_p(\mathbb{R}^3).$

\subsubsection{Linear approximating equation}
We first prove the existence of axisymmetric solutions to \eqref{eq.reg eq} for each fixed $n\ge n_0$, following the classical proofs \cite{MR1055522,MR1737547}. To this end, we consider the following linear approximating equation with the initial data $f^{n}_0(v) \in C^{\infty}_c(\mathbb{R}^3)$ defined in \eqref{initial}:
\begin{equation}\label{3.5}
\begin{aligned}
    \partial_t f &= Q^{\varepsilon}(g,f) + \varepsilon \Delta f =\overline{a_{ij}}^{\varepsilon,g} \partial_i \partial_j f - \overline{c}^{\varepsilon,g} f+ \varepsilon\Delta f .
    \end{aligned}
\end{equation} Suppose that $g$ is axisymmetric and further satisfies $M(g)=M(f^n_0)>0$, \begin{equation}\label{pi_1}
\| g \|_{L^{\infty}(0,T;L^1\cap W^{2,\infty}(\mathbb{R}^3)) \cap W^{1,\infty}(0,T;W^{-2,1}(\mathbb{R}^3)))} \leq C_1,
\end{equation} such that \begin{equation}\label{pi_2}
0 \leq g(t,v) \leq C_2 e^{-C_3|v|^2}, \qquad t \in [0,T],\, x \in \mathbb{R}^3,
\end{equation} where $C_1$, $C_2,$ and $C_3$ are positive constants. Thanks to our regularization, we can easily see that the coefficients of the operator in \eqref{3.5} are smooth (and hence are in the H\"older class $\mathcal{H}^{l_1,l_2}_{t,x}$ for any $l_1,l_2\in \mathbb{N}$). Then by the classical parabolic results \cite{MR108647, MR181836,MR1055522,MR1737547} we see that there exists a unique smooth solution $f$ to \eqref{3.5} and \eqref{initial}. We claim that this $f$ satisfies the bound \eqref{pi_2} and hence \eqref{pi_1}. The axisymmetricity of $f$ will be proved in the next subsection below.
\subsubsection{Non-negativity and upper-bound estimate}
The proof of \eqref{pi_2} for $f$ with a better Gaussian-type lower bound was given in \cite{MR1055522} via the standard maximum and minimum principles for the linear parabolic equation \eqref{arsenev eq}. Regarding the linear equation \eqref{3.5} in our case, we provide here the proof of non-negativity of $f$ and a formation of a Gaussian upper-bound given that $g$ satisfies \eqref{pi_2}.

First of all, the non-negativity of $f$ is obtained by the Feynmann-Kac formula for the solution representation of $f$ to the linear parabolic equation \eqref{3.5} with strict ellipticity given that $f$ is initially non-negative. Regarding a Gaussian upper-bound of $f,$ we consider an additional weight $ w(v) = e^{c_0 |v|^2} $ for a sufficiently large $c_0>0$ and derive an equation for \( h(t, v) = w(v) f(t, v) \). Note that the linear equation \eqref{3.5} for $f$ can be written as  
\[
\partial_t f = A_{ij} \partial_i \partial_j f + C f.
\]
Since \( h(t, v) = w(v) f(t, v) \), we observe that
\[
\partial_i f = \partial_i \left( \frac{h}{w} \right) = \frac{\partial_i h}{w} - \frac{h \partial_i w}{w^2},
\]and 
\[
\partial_i \partial_j f = \partial_i \partial_j \left( \frac{h}{w} \right) = \frac{\partial_i \partial_j h}{w} - \frac{\partial_i h \partial_j w}{w^2} - \frac{\partial_j h \partial_i w}{w^2} - \frac{h \partial_i \partial_j w}{w^2} + \frac{2h (\partial_i w)(\partial_j w)}{w^3},
\]and
\[
\partial_t f = \partial_t \left( \frac{h}{w} \right) = \frac{\partial_t h}{w}.
\]
     Spatial derivatives:
Substitute \( \partial_i \partial_j f \):
\[
A_{ij} \partial_i \partial_j f = A_{ij} \left( \frac{\partial_i \partial_j g}{w} - \frac{\partial_i g \partial_j w}{w^2} - \frac{\partial_j g \partial_i w}{w^2} - \frac{g \partial_i \partial_j w}{w^2} + \frac{2g (\partial_i w)(\partial_j w)}{w^3} \right).
\]
Thus, $h $ satisfies the following equation:
\[
\partial_t h = A_{ij} \left( \partial_i \partial_j h - \frac{\partial_i h \partial_j w}{w} - \frac{\partial_j h \partial_i w}{w} - \frac{h \partial_i \partial_j w}{w} + \frac{2h (\partial_i w)(\partial_j w)}{w^2} \right) + C h.
\]Here, for \( w(v) = e^{c_0 |v|^2} \), the derivatives of the weight $w$ are calculated as \( \partial_i w = 2 c_0 v_i w \),
 \( \partial_i \partial_j w = 2 c_0 \delta_{ij} w + 4 c_0^2 v_i v_j w \),
 \( \frac{\partial_i w}{w} = 2 c_0 v_i \),
\( \frac{\partial_i \partial_j w}{w} = 2 c_0 \delta_{ij} + 4 c_0^2 v_i v_j \),
and \( \frac{(\partial_i w)(\partial_j w)}{w^2} = 4 c_0^2 v_i v_j \).
Collecting these, we finally obtain
\begin{equation}\label{eq for h}
\partial_t h = A_{ij} \partial_i \partial_j h - 4 c_0 A_{ij} v_i \partial_j h - 2 c_0 (\operatorname{tr} A) h + C h,
\end{equation}
where \( \operatorname{tr} A = \sum_{i=1}^n A_{ii} \) is the trace of \( A_{ij} \).
Since $A=\bar{a}_{i,j}^{\varepsilon,g} +\varepsilon\text{Id},$ we have that the trace $\operatorname{tr} A $ is strictly positive. Then for a sufficiently large $c_0,$ we have $- 2 c_0 (\operatorname{tr} A)  + C <0.$ Then by the standard maximum principle for a linear parabolic equation \eqref{eq for h} with a negative coefficient of the zeroth-order term $Ch$, we have 
$h(t,v)\le \sup_{v\in\mathbb{R}^3} h(0,v) = \sup_{v\in\mathbb{R}^3} e^{c_0|v|^2} f^n_0\le C,$ for some $C$ depending only on $n$ and $M(f_0)$. Therefore, we have 
$$f(t,v)\le C e^{-c_0|v|^2}.$$ 
Together with the non-negativity of $f,$ we have \eqref{pi_2} for $f$ (for $C_2$ and $C_3$ chosen to be $C$ and $c_0$).

\subsubsection{Parabolic regularity and Schauder's fixed point theorem}Recall that we are approximating the nonlinear equation \eqref{eq.reg eq} and its solution by means of the linear approximating equations \eqref{3.5} and its solution sequence $f$. The coefficients of \eqref{3.5} are smooth for each fixed $\varepsilon>0$ and are at least in the H\"older class $\mathcal{H}^{l/2,l}_{t,v}$ for $l\in (1,2)$ for any $\varepsilon>0$. Then by the classical results on a generic linear parabolic equation \cite{MR1055522,MR199563,Ladyzhenskaya_2003,ladyzhenskaia1968linear,MR181836}, we have $f\in \mathcal{H}^{l/2+1,l+2}_{t,v}$ for $l\in (1,2)$ given that the initial data is in $\mathcal{H}^{l}_{v}$ (which is also true). In addition, we have already observed that $f$ satisfies \eqref{pi_1} and \eqref{pi_2}.  Then given that $f$ is axisymmetric (which will be proved in the next subsection), by applying Schauder’s fixed point theorem for $f=\Psi(g)$ in the solution space defined via \eqref{pi_1}-\eqref{pi_2} and $M(f)=M(f^n_0)$, we obtain an axisymmetric solution $f^{\varepsilon}$ to \eqref{eq.reg eq} that satisfies all the above properties. 
%\begin{equation}
%    \label{ansatz}M(f^{(n)}(t)) =M(f^{(n)}(0))\text{ and } \sup_{t \in [0,T]} \| f^{(n)}(t) \|_{H^k} \leq C(k,\varepsilon,T),\end{equation} for $ k \in \mathbb{N} \cup \{0\}, \,T>0,$
Note that $f$ is axisymmetric that will be shown later and the upper bounds of $f$ in these spaces is only depending on $C$. 
%, by integrating \eqref{eq.reg eq}, \begin{multline*}
%    \frac {d}{dt} M(f^{(n+1)}(t))=\frac {d}{dt} \int_{\mathbb{R}^3} f^{(n+1)}(t) \,\mathrm{d}v \\= \int_{\mathbb{R}^3} \lambda(\varepsilon v) \Delta f^{(n+1)} \,\mathrm{d}v - \int_{\mathbb{R}^3}  f^{(n+1)}\Delta \lambda(\varepsilon v) \,\mathrm{d}v = 0.
%\end{multline*} Therefore, we have $M(f^{(n+1)}(t))=M(f^{(n+1)}_0).$ In addition, by the uniform ellipticity of regularized collision operator, we obtain the regularity estimates $$\sup_{t \in [0,T]} \| f^{(n+1)}(t) \|_{H^k} \leq C(k,\varepsilon,T,M(f^{(n)}(t))), \qquad k \in \mathbb{N} \cup \{0\}, \,T>0.$$
%Here we note that 
%$$M(f^{(n)}(t)) =M(f^{(n)}(0))= M(f^\varepsilon_0)\le 1,$$and hence $M(f^{(n)}(t))$ and $C(k,\varepsilon,T,M(f^{(n)}(t)))$ are indeed uniform in $n$. 
%Therefore, the bootstrap ansatz \eqref{ansatz} holds for any $n \in \mathbb{N}\cup\{0\}$ and each $f^{(n)}$ exists.
\subsection{$f^{\varepsilon}$ is axisymmetric.} We further show that $f$ which solves the linear approximating equation is axially symmetric, so is $f^{\varepsilon}$ by the above procedure. To this end, we claim that $f(Tv)$ satisfies the linear system, where $T = T_{\theta}$ be a $3 \times 3$ rotation matrix such that $$T := \begin{pmatrix} \cos \theta & -\sin \theta & 0 \\
\sin \theta & \cos \theta & 0 \\
0 & 0 & 1 \end{pmatrix}, \qquad \theta \in [0,2\pi).$$ We observe that\begin{equation*}
	\begin{gathered}
		\varepsilon (\Delta f)(Tv)  
		= \varepsilon\Delta (f(Tv)) , \text{ for }\theta \in [0,2\pi).
	\end{gathered}
\end{equation*} In addition, one can show that $$Q^{\varepsilon}(g,f)(Tv) = Q^{\varepsilon}(g,f(Tv))(v)$$ as follows. In the rest of this subsection, we abbreviate and remove the notation $\varepsilon.$ 
Note that
\begin{equation}
\begin{gathered}\label{QT1}
    Q(g,f)(Tv) = \int_{\mathbb{R}^3} \partial_i a_{ij}(Tv-v_*) \left( g(v_*)(\partial_jf)(Tv) - f(Tv)\partial_jg(v_*) \right) \mathrm{d} v_* \\
    + \int_{\mathbb{R}^3} a_{ij}(Tv-v_*) \left( g(v_*)(\partial_i \partial_jf)(Tv) - (\partial_i f)(Tv)\partial_jg(v_*) \right) \mathrm{d} v_*
\end{gathered}
\end{equation} and 
\begin{equation}
\begin{gathered}\label{QT2}
    Q(g,f(Tv))(v) = \int_{\mathbb{R}^3} \partial_i a_{ij}(v-v_*) \left( g(v_*)\partial_j(f(Tv)) - f(Tv)\partial_jg(v_*) \right) \mathrm{d} v_* \\
    + \int_{\mathbb{R}^3} a_{ij}(v-v_*) \left( g(v_*)\partial_i \partial_j(f(Tv)) - \partial_i (f(Tv))\partial_jg(v_*) \right) \mathrm{d} v_*.
\end{gathered}
\end{equation} We first show that the first term of $Q(g,f)(Tv)$ in \eqref{QT1} and that of $Q(g,f(Tv))(v)$ in \eqref{QT2} are the same. Due to a special structure of the matrix $a$, we can see $(\nabla \cdot a)(Tz) = T \nabla \cdot a(z)$ for all $z \in \mathbb{R}^3$. Combining with the change of variables, we have
\begin{equation*}
\begin{gathered}
    \int_{\mathbb{R}^3} (\nabla \cdot a)(Tv-v_*) \cdot \left( g(v_*)(\nabla f)(Tv) - f(Tv)(\nabla g)(v_*) \right) \mathrm{d} v_* \\
    = \int_{\mathbb{R}^3} T \nabla \cdot a(v-v_*) \cdot \left( g(v_*)(\nabla f)(Tv) - f(Tv)(\nabla g)(Tv_*) \right) \mathrm{d} v_* \\
    = \int_{\mathbb{R}^3} \nabla \cdot a(v-v_*) \cdot T^{\top} \left( g(v_*)(\nabla f)(Tv) - f(Tv)(\nabla g)(Tv_*) \right) \mathrm{d} v_*.
\end{gathered}
\end{equation*} Using $T^{\top} = T^{-1}$ and
\begin{equation}\label{chain_2}
    \nabla (h(Mv)) = M^{\top} (\nabla h)(Mv), \qquad M \in \mathcal{M}_{3 \times 3}(\mathbb{R}),
\end{equation} it follows
\begin{equation*}
\begin{gathered}
    \int_{\mathbb{R}^3} \nabla \cdot a(v-v_*) \cdot T^{\top} \left( g(v_*)(\nabla f)(Tv) - f(Tv)(\nabla g)(Tv_*) \right) \mathrm{d} v_* \\
    = \int_{\mathbb{R}^3} \nabla \cdot a(v-v_*) \cdot \left( g(v_*)\nabla (f(Tv)) - f(Tv)\nabla (g(Tv_*)) \right) \mathrm{d} v_* \\
    = \int_{\mathbb{R}^3} \nabla \cdot a(v-v_*) \cdot \left( g(v_*)\nabla (f(Tv)) - f(Tv) \nabla g(v_*) \right) \mathrm{d} v_*.
\end{gathered}
\end{equation*} Now we claim
\begin{equation}\label{sym_est_front}
    \int_{\mathbb{R}^3} a_{ij}(Tv-v_*) g(v_*)(\partial_i \partial_jf)(Tv) \mathrm{d} v_* = \int_{\mathbb{R}^3} a_{ij}(v-v_*) g(v_*)\partial_i \partial_j(f(Tv)) \mathrm{d} v_*
\end{equation} and
\begin{equation}\label{sym_est_back}
    \int_{\mathbb{R}^3} a_{ij}(Tv-v_*) (\partial_i f)(Tv)\partial_jg(v_*) \mathrm{d} v_* = \int_{\mathbb{R}^3} a_{ij}(v-v_*) \partial_i (f(Tv))\partial_jg(v_*) \mathrm{d} v_*.
\end{equation} 

We prove \eqref{sym_est_front} first. Since the radial symmetric case is shown in \cite[Proposition~3.1]{MR3599518}, it should hold that
\begin{gather*}
    \sum_{j=1}^2 \sum_{i=1}^2\int_{\mathbb{R}^3} a_{ij}(Tv-v_*) g(v_*)(\partial_i \partial_jf)(Tv) \mathrm{d} v_* \\
    = \sum_{j=1}^2 \sum_{i=1}^2 \int_{\mathbb{R}^3} a_{ij}(v-v_*) g(v_*)\partial_i \partial_j(f(Tv)) \mathrm{d} v_*.
\end{gather*} The equality holds for $i=j=3$ clearly. Thus, it remains to show that 
\begin{multline*}
    \sum_{i=1}^2\int_{\mathbb{R}^3} a_{i3}(Tv-v_*) g(v_*)(\partial_i \partial_3f)(Tv) \mathrm{d} v_*\\ = \sum_{i=1}^2 \int_{\mathbb{R}^3} a_{i3}(v-v_*) g(v_*)\partial_i \partial_3(f(Tv)) \mathrm{d} v_*.
\end{multline*}
For simplicity, let $z = v-v_*$. Then, direct computation yields that 
\begin{gather*}
    \sum_{i=1}^2 a_{i3}(Tz) (\partial_i \partial_3f)(Tv) = -z_1z_3 |z|^{\gamma} \cos \theta  (\partial_1 \partial_3f)(Tv) + z_2z_3 |z|^{\gamma} \sin \theta  (\partial_2 \partial_3f)(Tv) \\
    - z_1z_3 |z|^{\gamma} \sin \theta  (\partial_1 \partial_3f)(Tv) - z_2z_3 |z|^{\gamma} \cos \theta  (\partial_2 \partial_3f)(Tv) = \sum_{i=1}^2 a_{i3}(z) \partial_i \partial_3(f(Tv)).
\end{gather*} Thus, \eqref{sym_est_front} is obtained from the calculation above and the change of variables. 

To prove \eqref{sym_est_back}, it suffices to see that $Tz \otimes Tz = T(z \otimes z)T^{\top}$ for all $z \in \mathbb{R}^3$. This can be shown by the simply relationship $$T(z \otimes z) T^{\top} = (Tz \otimes z)T^{\top} = (T (z \otimes Tz))^{\top} = (Tz \otimes Tz)^{\top} = Tz \otimes Tz. $$ Then, it follows $a(Tz) = T a(z) T^{\top}$, and we can have 
\begin{align*}
    a_{ij}(T(v-v_*)) (\partial_i f)(Tv)(\partial_jg)(Tv_*) &= (\nabla f)(Tv)^{\top} a(T(v-v_*)) (\nabla g)(Tv_*) \\
    &= \nabla(f(Tv))^{\top} T^{\top} a(T(v-v_*)) T \nabla(g(Tv_*)) \\
    &=T^{\top} a(T(v-v_*)) T \nabla g(v_*) \cdot \nabla(f(Tv)) \\
    &= a_{ij}(v-v_*) \partial_i(f(Tv)) \partial_j g(v_*).
\end{align*} This implies \eqref{sym_est_back}.

Therefore, we obtain a (unique) smooth axisymmetric solution $f^{\varepsilon}.$

\subsection{Uniform-in-$\varepsilon$ a priori estimates and passing to the limit $\varepsilon\to 0$}\label{sec.unif.epsilon}
We now derive uniform bounds for moments of the solution using test functions of the form 
 $(1 + |v|^2)^{p/2}$ for $p \ge  2$. These estimates are independent of 
the regularization parameter $\varepsilon$ and are essential for compactness. 
We first have from that the regularized system \eqref{eq.reg eq} is the divergence form $$M(f_t^{\varepsilon}) = M(f_0^{\varepsilon})=M(f^n_0)$$ for all $t \geq 0$. Multiplying both terms of \eqref{eq.reg eq} by $|v|^2$ and integrating over $\mathbb{R}^3$, we have \begin{equation*}
		\begin{gathered}
			\frac {d}{dt} \int_{\mathbb{R}^3} |v|^2 f^{\varepsilon}(t) \,\mathrm{d}v 
			= \int_{\mathbb{R}^3} |v|^2 Q^\varepsilon(f^\varepsilon, f^\varepsilon) \,\mathrm{d}v + \varepsilon\int_{\mathbb{R}^3} |v|^2 \Delta f^\varepsilon \,\mathrm{d}v .
		\end{gathered}
	\end{equation*}
Note from \eqref{def_bar} that the first integral on the right-hand side is equal to \begin{multline*}\int_{\mathbb{R}^3} |v|^2 \partial_i \left( \overline{a_{ij}}^{\varepsilon,f^{\varepsilon}} \partial_j f^{\varepsilon} - \overline{b}_{i}^{\varepsilon,f^{\varepsilon}} f^{\varepsilon} \right) \mathrm{d} v = \int_{\mathbb{R}^3} \overline{a_{ij}}^{\varepsilon,f^{\varepsilon}} \delta_{ij} f^{\varepsilon} \,\mathrm{d} v +2 \int_{\mathbb{R}^3} v_i \overline{b}_{i}^{\varepsilon,f^{\varepsilon}} f^{\varepsilon} \mathrm{d} v \\
			= 2 \int_{\mathbb{R}^3} v_i \left( (a_{ij}\partial_j \mu_\varepsilon) * f^{\varepsilon} \right) f^{\varepsilon} \mathrm{d} v = 2 \int_{\mathbb{R}^3} \int_{\mathbb{R}^3} v_i (a_{ij}\partial_j \mu_\varepsilon)(v-v_*) f^{\varepsilon}(v_*) f^{\varepsilon}(v) \,\mathrm{d}v_* \mathrm{d} v \\= 0,
		\end{multline*} where we used \eqref{def.muep} with $(|v|^2\delta_{ij} - v_iv_j)v_j = 0$ in the last equality. On the other hand, the remainder term is bounded by $C\varepsilon M(f_0^{\varepsilon})$ for some $C>0$ not depending on $\varepsilon$ via the integration by parts. Thus, we obtain $$\frac {d}{dt} \int_{\mathbb{R}^3} |v|^2 f^{\varepsilon}(t) \,\mathrm{d}v \leq C\varepsilon M(f^n_0),$$ and \begin{equation}\label{E_est} E(f^{\varepsilon}_t) \leq E(f^n_0) + tC\varepsilon  M(f^n_0)\end{equation} on the interval $[0,T]$ for any $T>0$. Moreover, $$\frac {d}{dt} \int_{\mathbb{R}^3} f^{\varepsilon}(t) \log f^{\varepsilon}(t) \,\mathrm{d}v \leq 0, $$ and thus, 
	\begin{equation}\label{H_est} \text{H}(f^{\varepsilon}_t) \leq \text{H}(f^n_0) \end{equation} holds for all $t \in [0,T]$ and $\varepsilon > 0$. From \eqref{E_est} and \eqref{H_est}, we deduce that $\{ f^{\varepsilon}_t \}_{\varepsilon > 0}$ is weakly compact in $L^1(\mathbb{R}^3_v)$ for each $t\in[0,T]$ by satisfying the Dunford-Pettis compactness criterion. Moreover, we can have from \eqref{eq.reg eq} the moments propagation properties: $$M_s(f^{\varepsilon}_t) \leq C(M_s(f^n_0))$$ for all $t \in [0,T]$ and $s > 0$.  This implies uniform equicontinuity in $t$; for any $\varphi \in C^2_{b,1}(\mathbb{R}^3)$, i.e., $C^2$ function with $\| \varphi \|_{L^{\infty}} + \| \nabla \varphi \|_{L^{\infty}} + \| \nabla^2 \varphi \|_{L^{\infty}} \leq 1,$  we have\begin{equation}
		\begin{aligned}
			&\int_{\mathbb{R}^3} \varphi (f^{\varepsilon}_t - f^{\varepsilon}_{t'}) \,\mathrm{d}v \\
			&\hphantom{\qquad}\leq C\int_{t'}^t \int_{\mathbb{R}^3} \int_{\mathbb{R}^3} (1+|v|^{2+\gamma} + |v_*|^{2+\gamma}) f^{\varepsilon}_{\tau}(v) f^{\varepsilon}_{\tau}(v_*) \,\mathrm{d}v_* \mathrm{d}v \mathrm{d}\tau + (t-t')\varepsilon M(f^{\varepsilon}_0) \\
			&\hphantom{\qquad}\leq C(t-t').
		\end{aligned}
	\end{equation} Therefore, by the Arzelà-Ascoli theorem, we can deduce that there exists a subsequence $\{ f^{\varepsilon'} \}_{\varepsilon' > 0} \subset \{ f^{\varepsilon} \}_{\varepsilon > 0}$ and a limit $f^n \in L^{\infty}(0,T;L^1_{s} \cap L \log L(\mathbb{R}^3))$ which is a weak solution to \eqref{eq:L} with the initial data $f^n_0$ defined in \eqref{initial}. We note that this solution $f^n$ satisfies all conditions for Proposition \ref{prop.1.3}, and thus by the ellipticity of $\overline{a_{ij}(v)}$, the smoothness of the initial data $f^n_0$,
     and that $f^n$ satisfies the Gaussian upper-bound \eqref{pi_2}, we conclude that $f^n \in C^{\infty}((0,\infty);\mathcal{S}(\mathbb{R}^3))$ by the smoothing effect. For the proof of smoothing effects, see the proof of Theorem 5 of \cite[page 227-237]{MR1737547}.
	
\subsection{Axi-symmetry Preservation}\label{sec.axisym}
To show that the solution $f^n$ remains axisymmetric, we verify that the collision operator $Q(f^n, f^n)$ preserves the axisymmetric structure. For an axisymmetric test function $\phi(v) = \phi(T_\theta v)$, the rotational invariance of $a_{ij}(v - v_*)$ implies:
\[
\int_{\mathbb{R}^3} Q(f^n, f^n) \phi dv = \int_{\mathbb{R}^3} Q( f^n(T_\theta v), f^n(T_\theta v)) \phi dv,
\]
where $T_\theta$ denotes the rotation operator around the axis of symmetry. Since the initial data is axisymmetric, it follows that $f^n(t, v) = f^n(t, T_\theta v)$ for all $t > 0$.

\subsection{Passing to the limit as $n \to \infty$}
Given the solution $f^n \in C^{\infty}((0,\infty);\mathcal{S}(\mathbb{R}^3))$ to the Landau equation \eqref{eq:L} for each smooth compactly supported initial data $f^n_0$ defined in \eqref{initial}, we pass to the limit as $n\to \infty. $ We will prove that there exists a limit $f \in C((0,\infty);\mathcal{P}(\mathbb{R}^3))$ which solves \eqref{eq:L} with the initial data $f_0\in \mathcal{P}_p(\mathbb{R}^3)$ for $p\ge 2$ in the weak sense (Definition \ref{weak def}). To this end, we have to prove the sufficient conditions for the use of Arzela-Ascoli theorem. We introduce a brief outline based on the proof of Theorem 3 of \cite{FOURNIER20211961}. Note that our initial data \eqref{initial} are mollified version of the cutoff initial data $\alpha_n^{-1} \mathbf{1}_{|v| \leq n} f_0(dv)$ of \cite{FOURNIER20211961}, which makes \eqref{initial} smoother.

For $p>2,$ we have for any $T>0,$ there exists a constant $C_T>0$ such that for any $2 < s \leq p$
\begin{equation}
    \label{gain of moment 1}\sup_{t\in[0,T]}	M_s(f^n(t)) +  \int_0^T M_{s+\gamma}(f^n(\tau)) \mathrm{d}\tau \leq C_T,
\end{equation}	by \eqref{gain_p}. On the other hand, if $p=2, $ we use the de La Vall\'ee Poussin theorem and guarantee the existence of a $C^2$ function $h:[0,\infty)\to [0,\infty)$ such that $h''\in [0,1],$ $h'(0)=1$, $h'(\infty)=\infty$ and $\int_{\mathbb{R}^3}h(|v|^2)f_0(dv)<\infty.$ Then by \cite[Eq. (29)]{FOURNIER20211961}, there exists a finite constant $K_T$ for any $T>0$ such that for all $n\ge n_0,$ 
\begin{equation}
    \label{gain of moment 2}\sup_{t\in[0,T]}	\int_{\mathbb{R}^3}h(|v|^2)f^n_t(dv) +  \int_0^T \int_{\mathbb{R}^3}|v|^{2+\gamma}h'(|v|^2)f^n_t(dv)  \mathrm{d}t \leq K_T.
\end{equation}	

Then using the moment estimates, now we show that for each $t\ge 0,$ the family $\{f_t^n\}_{n\ge n_0}$ is relatively compact in $\mathcal{P}(\mathbb{R}^3). $   We note that, for each $n\ge n_0$, $$E(f^{n}_0) = \int_{\mathbb{R}^3} |v|^2 (\alpha_n^{-1} \mathbf{1}_{|v| \leq n} f_0(dv)) * \sigma_{n} = \int_{\mathbb{R}^3} (|v|^2 * \sigma_{n}) \alpha_n^{-1} \mathbf{1}_{|v| \leq n} f_0(dv),$$ and hence by the dominate convergence theorem, we obtain the last integral converges to $E(f_0)$.  Then using the energy conservation \eqref{energy con}, we have for each $n\ge n_0$ $$E(f^n_t)\le E(f^n_0)\le E(f_0)=:E_0. $$ Also, since the set $\{f\in \mathcal{P}(\mathbb{R}^3); E(f)\le E_0\}$ is compact for any $E_0>0$, we conclude that $\{f_t^n\}_{n\ge n_0}$ is relatively compact in $\mathcal{P}(\mathbb{R}^3). $ 

Then we can also prove that $\{f^n_t\}_{n\ge n_0}$ is equicontinuous in time; namely, for the weak distance $d_{\text{weak}}$ which is defined as in \cite{FOURNIER20211961}
$$d_{\text{weak}}(f,g)\eqdef \sup_{\varphi\in C^2_{b,1}}\left|\int_{\mathbb{R}^3}\varphi(v)(f-g)(dv)\right|, $$ we prove that $$\lim_{\epsilon\to 0}\sup_{n\ge n_0}\sup_{s,t\in[0,T],|t-s|\le \epsilon}d_{\text{weak}} (f^n_t,f^n_s)=0.$$ Here $C^2_{b,1}(\mathbb{R}^3)$ is defined as $$C^2_{b,1}(\mathbb{R}^3)=\{\varphi\in C^2(\mathbb{R}^3):\| \varphi \|_{L^{\infty}} + \| \nabla \varphi \|_{L^{\infty}} + \| \nabla^2 \varphi \|_{L^{\infty}} \leq 1\}.$$ From the weak formulation in \eqref{weak def} with the estimates that $|a_{ij}(z)|\le |z|^{2+\gamma}$ and $|b_i(z)|\le |z|^{1+\gamma},$ we obtain 
$$d_{\text{weak}} (f^n_t,f^n_s)\lesssim \int_s^t  \int_{\mathbb{R}^3}(1+|v|^{\gamma+2})f^n_\tau(dv)d\tau.$$ Splitting the case with $|v|\le A$ and $|v|\ge A$ and using \eqref{gain of moment 1} for $p>2$ case and \eqref{gain of moment 2} for $p=2$, we obtain
\begin{equation}d_{\text{weak}} (f^n_t,f^n_s)\lesssim\begin{cases}
    &(1+A^{\gamma+2})\epsilon + \frac{C_T}{A^{s+\gamma-2}},\text{ for }p>2,\\
    &(1+A^{\gamma+2})\epsilon + \frac{K_T}{h'(A^2)},\text{ for }p=2.
\end{cases}\end{equation} For both cases, for any given small $\eta>0, $ we choose $A$ sufficiently large such that 
$\frac{C_T}{A^{s+\gamma-2}},\frac{K_T}{h'(A^2)}\ll \frac{\eta}{2}.$ Then choosing $\epsilon$ sufficiently small such that $(1+A^{\gamma+2})\epsilon \ll \frac{\eta}{2}$, we can make $d_{\text{weak}} (f^n_t,f^n_s)<\eta$ for any sufficiently small $\eta>0.$ Therefore, by the Arzela-Ascoli theorem, we obtain that there exists a subsequence of $\{f^n_t\}_{n\ge n_0}$ which converges weakly to a limit $f\in C([0,\infty);\mathcal{P}(\mathbb{R}^3))$. This $f$ satisfies $E(f_t)\le E(f_0)$. Moreover, we have from \eqref{gain_p} and \eqref{t1t2} that $f\in L^{\infty}((0,\infty);\mathcal{P}_{p}(\mathbb{R}^3)) \cap L^{1}_{loc}((0,\infty);\mathcal{P}_{p+\gamma}(\mathbb{R}^3))$. One can easily check that this limit is indeed a weak solution to the Landau equation \eqref{eq:L} satisfying the weak formulation (Definition \ref{weak def}), since we can pass to the limit $n\to \infty$ in not just the linear terms but also the nonlinear terms by adding and subtracting $a_{ij}\partial_i\partial_j\varphi f^n_tf_t$ (and similarly $b_i\partial_i\varphi f^n_tf_t$) in the integrand of the nonlinear integrals and use the uniform $L^1$-moment bounds $E(f^n_t)\le E(f_0) $ and the fact that we test against $C^2_{b,1}$ functions. Lastly, our solution $f$ is also axisymmetric due to the same reasoning in Section \ref{sec.axisym}. This completes the proof of an axisymmetric measure-valued solution $f\in C([0,\infty);\mathcal{P}(\mathbb{R}^3))$ to the Landau equation \eqref{eq:L}.
\subsection{Analytic Regularity of Solutions}Indeed, our axisymmetric solution $f$ is analytic, thanks to Chen-Li-Xu's work (Theorem \ref{analytic theorem}) which states that any smooth solution with uniform $H^k$ bounds for all $k\in\mathbb{N}$ in the hard potential case is analytic for all $t>0.$ Indeed, as long as our solution (and hence the initial datum) is not a Dirac mass, then by Lemma \ref{lem_line}, we have the instantaneous spanning of the support escaping a line support. Then by the axi-symmetry of solution, we have the support instantaneously containing a triangle. 
\subsubsection{Case with $p>2$}Then in the case when $p>2$, we follow Desvillettes-Villani's work (Theorem \ref{Villani theorem}) on the support containing three balls in \cite{MR1737547} in the hard potential case and show that our axisymmetric solution is indeed bounded in $H^k$ for any $k\in \mathbb{N}.$ Then by Theorem \ref{analytic theorem} (or Theorem \ref{fournier analytic} by Fournier-Heydecker), we conclude that our solution is indeed analytic in the case $p>2$. 
\subsubsection{Case with $p=2$}
If $p=2,$ then we follow the proof of Theorem \ref{fournier analytic} in \cite[Theorem 5]{FOURNIER20211961}. For the self-containedness, we sketch the proof here. Namely, for $f_0\in \mathcal{P}_2(\mathbb{R}^3)$ axisymmetric and not being a Dirac mass, let $f_t$ be any of corresponding axisymmetric weak solutions starting at $f_0.$ Then by Theorem \ref{Villani theorem}, for a fixed $t_0>0,$ we have $M_{2+\delta}(f_{t_1})<\infty$ for any arbitrary $t_1\in (0,t_0).$ $f_{t_1}$ is also not a Dirac mass due to the conservation of energy and momentum. Then by Lemma \ref{lem_line}, we can find $t_2\in [t_1,t_0)$ such that $f_{t_2}$ is not concentrated on a line and $M_{2+\delta}(f_{t_2})<\infty$. Then by Theorem \ref{Villani theorem} there exists a weak solution $g_t$ with the initial datum $g_0\eqdef f_{t_2} $ such that $\sup_{t\ge \epsilon}\|g_t\|_{H^k(\mathbb{R}^3)}<\infty$ for any $k\in\mathbb{N}$ and $\epsilon>0.$ Since $f_{t_2}$ is axisymmetric, the solution $g$ is also axisymmetric. By Theorem \ref{analytic theorem}, $g_t$ is analytic for all $t>0.$ Finally, by the uniqueness theorem (Theorem \ref{uniq.thm}) for the solutions with $\mathcal{P}_{2+\delta}$ initial data, $f_{t_0}=g_{t_0-t_2}$ is analytic. This completes the proof.

\appendix
\section{Proof of the gain of moments for hard potentials}\label{app.gain of M} The following proof is from \cite[Section 3]{MR1737547}. Here we introduce a sketch of it.
From the Landau equation \eqref{eq:L}, we have for $s>2$ that
	\begin{equation*}
	\begin{gathered}
		\frac {d}{dt} M_s(f(t)) = \int_{\mathbb{R}^3} \overline{a_{ij}(v)} f(v) \partial_i \partial_j (1+|v|^2)^{\frac {s}{2}} \,\mathrm{d}v \\
		+ 2 \int_{\mathbb{R}^3} \overline{b_i(v)} f(v) \partial_i (1+|v|^2)^{\frac {s}{2}} \mathrm{d}v \\
		= s\int_{\mathbb{R}^3} \overline{a_{ii}(v)} f(v) (1+|v|^2)^{\frac {s-2}{2}} \,\mathrm{d}v  \quad (=:I)\\
		+ s(s-2) \int_{\mathbb{R}^3} \overline{a_{ij}(v)} f(v) v_i v_j (1+|v|^2)^{\frac {s-4}{2}} \,\mathrm{d}v \quad (=: II)\\
		+ 2s \int_{\mathbb{R}^3} \overline{b_i(v)} f(v) v_i (1+|v|^2)^{\frac {s-2}{2}} \mathrm{d}v \quad (=: III).
	\end{gathered}
\end{equation*} Here we used $\partial_i^2 (1+|v|^2)^{\frac {s}{2}} = s(1+|v|^2)^{\frac {s-2}{2}} + s(s-2) v_i^2 (1+|v|^2)^{\frac {s-4}{2}}$.
	We now recall the definition of $\overline{a_{ij}}$ and $\overline{b_i}$:
	$$\overline{a_{ij}(v)} = \int_{\mathbb{R}^3} \left( \delta_{ij} - \frac {(v-v_*)_i(v-v_*)_j}{|v-v_*|^2} \right)|v-v_*|^{2+\gamma} f(v_*) \,\mathrm{d}v_*,$$
	$$\overline{a_{ii}(v)} = 2\int_{\mathbb{R}^3} |v-v_*|^{2+\gamma} f(v_*) \,\mathrm{d}v_*,$$
	\begin{equation*}
	\begin{aligned}
		\overline{b_{i}(v)} &= -2 \int_{\mathbb{R}^3} (v-v_*)_i|v-v_*|^{\gamma} f(v_*) \,\mathrm{d}v_* \\
		&= -2 \int_{\mathbb{R}^3} v_i |v-v_*|^{\gamma} f(v_*) \,\mathrm{d}v_* + 2 \int_{\mathbb{R}^3} {v_*}_i |v-v_*|^{\gamma} f(v_*) \,\mathrm{d}v_* .
	\end{aligned}
\end{equation*}
Then we can have
	\begin{equation*}
	\begin{gathered}
		I = 2s\int_{\mathbb{R}^3} \int_{\mathbb{R}^3} |v-v_*|^{2+\gamma} (1+|v|^2)^{\frac {s-2}{2}} f(v) f(v_*) \,\mathrm{d}v_* \mathrm{d}v \\
		= 2s\int_{\mathbb{R}^3} \int_{\mathbb{R}^3} |v-v_*|^{\gamma} (|v|^2 - 2v \cdot v_* + |v_*|^2) (1+|v|^2)^{\frac {s-2}{2}} f(v) f(v_*) \,\mathrm{d}v_* \mathrm{d}v,
	\end{gathered}
\end{equation*} and
	\begin{equation*}
	\begin{gathered}
		III = -4s\int_{\mathbb{R}^3} \int_{\mathbb{R}^3} |v-v_*|^{\gamma}(v-v_*)_iv_i (1+|v|^2)^{\frac {s-2}{2}} f(v) f(v_*) \,\mathrm{d}v_* \mathrm{d}v \\
		= -4s \int_{\mathbb{R}^3} \int_{\mathbb{R}^3} |v-v_*|^{\gamma} |v|^2 (1+|v|^2)^{\frac {s-2}{2}} f(v) f(v_*) \,\mathrm{d}v_* \mathrm{d}v \\
		+ 4s\int_{\mathbb{R}^3} \int_{\mathbb{R}^3} |v-v_*|^{\gamma} v \cdot v_* (1+|v|^2)^{\frac {s-2}{2}} f(v) f(v_*) \,\mathrm{d}v_* \mathrm{d}v,
	\end{gathered}
\end{equation*} thus,
		\begin{equation*}
	\begin{gathered}
		I + III = -2s \int_{\mathbb{R}^3} \int_{\mathbb{R}^3} |v-v_*|^{\gamma} |v|^2 (1+|v|^2)^{\frac {s-2}{2}} f(v) f(v_*) \,\mathrm{d}v_* \mathrm{d}v \\
		+ 2s \int_{\mathbb{R}^3} \int_{\mathbb{R}^3} |v-v_*|^{\gamma} |v_*|^2 (1+|v|^2)^{\frac {s-2}{2}} f(v) f(v_*) \,\mathrm{d}v_* \mathrm{d}v.
	\end{gathered}
\end{equation*}
	On the other hand, using $$\delta_{ij} |v-v_*|^2 v_iv_j - (v-v_*)_i(v-v_*)_jv_iv_j = |v|^2|v_*|^2 - (v \cdot v_*)^2,$$ we can have
	\begin{equation*}
	\begin{gathered}
		II = s(s-2) \int_{\mathbb{R}^3} \int_{\mathbb{R}^3} (\delta_{ij}|v-v_*|^2 - (v-v_*)_i(v-v_*)_j)v_i v_j \\
		|v-v_*|^{\gamma} (1+|v|^2)^{\frac {s-4}{2}} f(v) f(v_*) \,\mathrm{d}v \\
		= s(s-2) \int_{\mathbb{R}^3} \int_{\mathbb{R}^3} (|v|^2|v_*|^2 - (v \cdot v_*)^2) |v-v_*|^{\gamma} (1+|v|^2)^{\frac {s-4}{2}} f(v) f(v_*) \,\mathrm{d}v \\
		= s(s-2) \int_{\mathbb{R}^3} \int_{\mathbb{R}^3} \frac{|v|^2|v_*|^2 - (v \cdot v_*)^2}{1+|v|^2} |v-v_*|^{\gamma} (1+|v|^2)^{\frac {s-2}{2}} f(v) f(v_*) \,\mathrm{d}v \\
		\leq s(s-2) \int_{\mathbb{R}^3} \int_{\mathbb{R}^3} |v_*|^2 |v-v_*|^{\gamma} (1+|v|^2)^{\frac {s-2}{2}} f(v) f(v_*) \,\mathrm{d}v.
	\end{gathered}
\end{equation*}
	Combining the above, we have
	\begin{equation*}
	\begin{gathered}
		I + II + III \leq -2s \int_{\mathbb{R}^3} \int_{\mathbb{R}^3} |v-v_*|^{\gamma} |v|^2 (1+|v|^2)^{\frac {s-2}{2}} f(v) f(v_*) \,\mathrm{d}v_* \mathrm{d}v \\
		+ s^2 \int_{\mathbb{R}^3} \int_{\mathbb{R}^3} |v-v_*|^{\gamma} |v_*|^2 (1+|v|^2)^{\frac {s-2}{2}} f(v) f(v_*) \,\mathrm{d}v_* \mathrm{d}v.
	\end{gathered}
\end{equation*} Using $s>2$ and symmetry give
\begin{equation*}
	\begin{gathered}
		\frac {d}{dt} M_s(f(t)) \leq s\int_{\mathbb{R}^3} \int_{\mathbb{R}^3} |v-v_*|^{\gamma} (-(1+|v|^2)^{\frac {s}{2}} - (1+|v_*|^2)^{\frac {s}{2}} \\
		+ \frac {s}{2} (1+|v|^2) (1+|v_*|^2)^{\frac {s-2}{2}} + \frac {s}{2} (1+|v_*|^2) (1+|v|^2)^{\frac {s-2}{2}}) f(v) f(v_*) \,\mathrm{d}v_* \mathrm{d}v.
	\end{gathered}
\end{equation*}
	Note that there exists $K>0$ and $C>0$ depending only on $s>2$ such that
	\begin{equation*}
	\begin{aligned}
		-\xi^s - \xi_*^s + \frac {s}{2} \xi^2\xi_*^{s-2} + \frac {s}{2} \xi_*^2 \xi^{s-2} \leq -K \xi^s + C(\xi\xi_*^{s-1} + \xi_* \xi^{s-1}).
	\end{aligned}
\end{equation*}
	Then,
	\begin{equation*}
	\begin{gathered}
		\frac {d}{dy} M_s(f(t)) \leq -K \int_{\mathbb{R}^3} \int_{\mathbb{R}^3} |v-v_*|^{\gamma} (1+|v|^2)^{\frac {s}{2}}  f(v) f(v_*) \,\mathrm{d}v_* \mathrm{d}v \\
		+ C \int_{\mathbb{R}^3} \int_{\mathbb{R}^3} |v-v_*|^{\gamma} ((1+|v|^2)^{\frac {1}{2}} (1+|v_*|^2)^{\frac {s-1}{2}} + (1+|v_*|^2)^{\frac {1}{2}} (1+|v|^2)^{\frac {s-1}{2}}) \\
		f(v) f(v_*) \,\mathrm{d}v_* \mathrm{d}v.
	\end{gathered}
\end{equation*}
With $-|v-v_*|^{\gamma} \leq -(1+|v|^2)^{\frac{\gamma}{2}} + C(1+|v_*|^2)^{\frac{\gamma}{2}}$, we have
\begin{equation*}
	\begin{gathered}
		\frac {d}{dt} M_s(f) \leq -K M(f_0) M_{s+\gamma}(f) + CK M_\gamma(f) M_s(f) \\
		+ CM_{\gamma+1}(f) M_{s-1}(f) + CM_1(f) M_{s+\gamma-1}(f).
	\end{gathered}
\end{equation*} Using interpolation inequalities and Young's inequality, we have 
\begin{equation}\label{gain}
	\frac {d}{dt} M_s(f) \leq - \frac {1}{2} K M(f_0) M_{s+\gamma}(f) + C(M(f_0) + E(f_0)) M_s(f)
\end{equation} and
	\begin{equation}\label{gain_p}
	\begin{aligned}
		M_s(f(t)) + K' \int_0^t M_{s+\gamma}(f(\tau)) \mathrm{d}\tau \leq C + C e^{Ct}.
	\end{aligned}
\end{equation}
	On the other hand, inserting
	\begin{equation*}
	\begin{aligned}
		M_{s}(f) \leq M(f_0)^{\frac {\gamma}{s+\gamma}} M_{s+\gamma}(f)^{\frac {s}{s+\gamma}}
	\end{aligned}
\end{equation*} into \eqref{gain} gives $$\frac {d}{dt} M_s(f) \leq - \frac {1}{2} K M(f_0)^{1-\frac {\gamma}{s}} M_{s}(f)^{\frac {s+\gamma}{s}} + C(M(f_0) + E(f_0)) M_s(f).$$
Therefore, there exists $C>0$ depending on $s$, $\gamma$, $M(f_0)$, and $E(f_0)$ such that
\begin{equation}\label{t1t2}
	M_s(f(t)) \leq \max \{ C, M_s(f(t_0)) \}, \qquad t \geq t_0.
\end{equation}
%
% BibTeX users please use
% these commends below are for a bibtex bibliography
%\bibliographystyle{plain}
%\bibliographystyle{abbrv}
%\bibliographystyle{acm}
%\bibliographystyle{alpha}
%\bibliographystyle{apalike}
%\bibliographystyle{ieeetr}
%\bibliographystyle{siam}
%\bibliographystyle{unsrt}
%\bibliographystyle{plainnat}
%\bibliographystyle{plainurl}
%\bibliographystyle{abbrvnat}
%\bibliographystyle{unsrtnat}
%\bibliographystyle{amsalpha.bst}  % this one is cool with amsrefs
%\bibliographystyle{amsplain.bst}
%% the ones below need to be called with AMSREFS package
%\bibliographystyle{amsrn.bst}  %this one is the default if not style is called
%\bibliographystyle{amsru.bst}
%\bibliographystyle{amsra.bst}
%\bibliographystyle{amsry.bst}
%\bibliographystyle{amsrs.bst}
%\bibliographystyle{amsxport.bst}

%%%% the bibliography styles below support arXiv eprint links
% \bibliographystyle{amsplain2link}
\bibliographystyle{amsplain3links}
\bibliography{bibliography.bib}{}

%%%%%%%%%%%%%%%%%%%%%%%%%%%%%%%%%%%%%%%%%%%%%%%%%%%%%%%%%%%%%%%%%%%%%%%%%%%%%%%%%%

%%% copy and paste the BBL file below here when you are ready to upload, and comment the above bibtex library
%\bigskip

\end{document}